\documentclass{book}
\usepackage[contrib,lang=british]{ems-book} 
\usepackage{tikz-cd}

\newtheorem{theorem}{Theorem}[section]
\newtheorem{definition}{Definition}[section]
\newtheorem{example}{Example}[section]
\newtheorem{conjecture}{Conjecture}[section]
\newtheorem{question}{Question}[section]
\newtheorem{proposition}{Proposition}[section]
\newtheorem{problem}{Problem}[section]

\newtheorem{lemma}{Lemma}[section]

\usepackage{tikz}
\usetikzlibrary{decorations.pathreplacing, calligraphy, calc, arrows.meta, decorations.pathmorphing}

\begin{document}
\mainmatter

\renewcommand{\mod}{\operatorname{mod}}
\newcommand{\proj}{\operatorname{proj}}
\newcommand{\cov}{\operatorname{cov}}
\newcommand{\cocov}{\operatorname{cocov}}
\newcommand{\inj}{\operatorname{inj}}
\newcommand{\prinj}{\operatorname{prinj}}
\newcommand{\Ext}{\operatorname{Ext}}
\newcommand{\row}{\operatorname{row}}
\newcommand{\im}{\operatorname{Im}}
\newcommand{\End}{\operatorname{End}}
\newcommand{\id}{\operatorname{id}}
\newcommand{\Hom}{\operatorname{Hom}}
\newcommand{\Tr}{\operatorname{Tr}}
\newcommand{\D}{\operatorname{D}}
\newcommand{\grade}{\operatorname{grade}}
\newcommand{\cograde}{\operatorname{cograde}}
\newcommand{\add}{\operatorname{\mathrm{add}}}
\renewcommand{\top}{\operatorname{\mathrm{top}}}
\newcommand{\rad}{\operatorname{\mathrm{rad}}}
\newcommand{\soc}{\operatorname{\mathrm{soc}}}
\newcommand{\ul}{\underline}
\newcommand{\red}{\textit{red}}

\renewcommand{\mod}{\operatorname{mod}}

\newcommand{\Z}{\mathbb{Z}}
\newcommand{\RHom}{\operatorname{\mathsf{R}Hom}}
\newcommand{\Lotimes}{\otimes^\mathsf{L}}
\newcommand{\Ocal}{\mathcal{O}}
\newcommand{\X}{\mathbb{X}}
\newcommand{\Ll}{\mathbb{L}}

\renewcommand{\ker}{\mathrm{Ker}}
\newcommand{\Db}{\mathrm{D}^{\mathrm{b}}}
\newcommand{\Simp}{\mathcal{S}}
\newcommand{\resdim}{\text{-}\mathrm{resdim}}
\newcommand{\coresdim}{\text{-}\mathrm{coresdim}}
\newcommand{\domdim}{\operatorname{domdim}}
\newcommand{\codomdim}{\operatorname{codomdim}}
\newcommand{\idim}{\operatorname{id}}
\newcommand{\Mod}{\operatorname{Mod}}
\newcommand{\pdim}{\operatorname{pd}}
\newcommand{\gldim}{\operatorname{gldim}}
\newcommand{\DynA}{\mathbb{A}}
\newcommand{\DynD}{\mathbb{D}}
\newcommand{\DynE}{\mathbb{E}}
\title{A survey on Auslander-Gorenstein algebras}
\titlemark{A survey on Auslander-Gorenstein algebras}

\emsauthor{1}{
	\givenname{Viktória}
	\surname{Klász}
	\mrid{1608805}
    }{V. Klász}
\emsauthor{2}{
	\givenname{Rene}
	\surname{Marczinzik}
	\mrid{1184000}
    }{R. Marczinzik}

\Emsaffil{1}{
	\pretext{}
	\department{Department of Mathematics}
	\organisation{University}
	\address{Endenicher Allee 60}
	\zip{53115}
	\city{Bonn}
	\country{Germany}
	\posttext{}
	\affemail{klasz@math.uni-bonn.de}
	\furtheremail{viktoria.klasz@gmail.com}}
\Emsaffil{2}{
	\pretext{}
	\department{University of Bonn}
	\organisation{University}
	\address{1}{Endenicher Allee 60}
	\zip{1}{53115}
	\city{1}{Bonn}
	\country{1}{Germany}
	\posttext{}
    \affemail{marczire@math.uni-bonn.de}
	\furtheremail{marczinzik.rene@googlemail.com}}
%

\classification[16G10]{16E10}

\keywords{Auslander-Gorenstein algebras, homological dimensions, monomial algebras, incidence algebras, Coxeter permutation}

\chapterdedication{Dedicated to the memory of Karl Libkowitz}

\begin{abstract}
We give a survey on Auslander-Gorenstein algebras with a focus on finite-dimensional algebras.
We put an emphasis on recent classification results for special classes of algebras and the newly discovered interactions of the Auslander-Reiten bijection with other well studied bijections in the literature.
\end{abstract}

\makecontribtitle


Let $R$ be a commutative local noetherian ring. Two of the most fundamental notions in commutative algebra and algebraic geometry are those of Gorenstein rings and regular rings.
$R$ is called \emph{Gorenstein} if the selfinjective dimension $\operatorname{id}_R R$ is finite and $R$ is called \emph{regular} if the Krull dimension of $R$ coincides with the embedding dimension $\operatorname{dim} m/m^2$ of $R$, where $m$ denotes the maximal ideal of $R$. A well known result due to Serre shows that $R$ is regular if and only if $R$ has finite global dimension \cite{Ser}.
In \cite{B}, Bass gives several equivalent characterisations of Gorenstein rings. For us, the most important one is the following:
\begin{theorem}
A commutative local noetherian ring $R$ is Gorenstein if and only if a minimal injective coresolution of the regular module $R$ has the following form:
$$0 \rightarrow R \rightarrow I^0 \rightarrow I^1 \cdots \rightarrow I^i \rightarrow \cdots $$
where $I^i \cong \bigoplus\limits_{\operatorname{ht}p=i}^{}{I(R/p)}$, the direct sum of all injective envelopes of the modules $R/p$ with $p$ a prime ideal of height $i$.
\end{theorem}

It can be shown that all the injective modules $I^i$ in the previous theorem have flat dimension at most $i$ and this characterises being Gorenstein, see \cite[Section 3]{FGR} and \cite{FR}.
This lead Auslander to the idea to generalise the notion of regular and Gorenstein rings to the non-commutative world as follows:

\begin{definition}
Let $R$ be a two-sided noetherian ring and $n \geq 1$. Then $R$ is called \emph{$n$-Gorenstein} 
if there exists an injective coresolution of the regular module 
$$0 \rightarrow R \rightarrow I^0 \rightarrow I^1 \rightarrow \cdots \rightarrow I^n \rightarrow \cdots$$
such that $\operatorname{flatdim} I^i \leq i$ for all $0 \leq i <n$.
$R$ is called \emph{Auslander-Gorenstein} if it is $n$-Gorenstein for all $n \geq 1$ and $\operatorname{idim} R < \infty$.
$R$ is called \emph{Auslander regular} if it is Auslander-Gorenstein and has finite global dimension.
\end{definition}
When a ring $R$ is $n$-Gorenstein for all $n$, then $R$ is said to satisfy the \emph{Auslander condition}. It is an open problem for finite-dimensional algebras  $R$ whether $R$ satisfying the Auslander condition implies that $R$ is Auslander-Gorenstein. We will discuss this conjecture later and relate it to some of the other classical homological conjectures. \newline
It seems that the definition of Auslander-Gorenstein rings was first stated explicitly in section 3 of \cite{FGR}, where the authors used the term "Gorenstein ring" instead of "Auslander-Gorenstein ring".
The notion of Auslander-Gorenstein rings quickly became popular in various areas of algebra.
We are aware of 3 textbooks with chapters on Auslander-Gorenstein rings, namely \cite{Bj1}, \cite{VO}, and \cite{LVO}.
There is also a survey article on Auslander-Gorenstein rings, namely \cite{Cl}.
All these sources focus on rings that are not necessarily finite-dimensional $K$-algebras, where $K$ is a field, with mostly infinite-dimensional examples in mind.
In this survey, we focus on the finite-dimensional side, where many important results have been obtained recently.
Therefore, we will only briefly look at general noetherian rings in the first section,  and then focus on the finite-dimensional theory.

\section{Auslander-Gorenstein rings for noetherian rings}
In this section we will mention several areas where Auslander-Gorenstein rings play an important role. Since we want to focus on finite-dimensional algebras in this survey, we will be very brief and mostly only give pointers to the literature.

\subsection{Non-commutative algebraic geometry and Artin-Schelter regular Algebras}
In Van Oystaeyen's approach to non-commutative algebraic geometry, Auslander regular algebras play a similar role as the classical commutative regular rings in classical algebraic geometry.
We refer to the textbook \cite{VO} for more details.

Another important class of rings in non-commutative ring theory and algebraic geometry are Artin-Schelter regular algebras. We refer for example to \cite{Rog} for a survey on Artin-Schelter regular algebras and the relevant definitions, and to \cite{Rog2} for their relevance in non-commutative algebraic geometry.
In \cite{Lev}, it was shown that graded algebras of polynomial growth that are Auslander regular are also Artin-Schelter regular.
The converse is formulated as an open question in 6.3.1 in \cite{Lev}, where some positive results are also obtained in special cases:
\begin{question}
Let $A=K \oplus A_1 \oplus A_2 \cdots $ be a finitely generated graded algebra over $K$.
Is $A$ an Artin-Schelter regular algebra if and only if $A$ is Auslander regular and is of polynomial growth?
\end{question}

\subsection{Enveloping algebras of finite-dimensional Lie algebras and quantum groups}
Universal enveloping algebras of finite-dimensional Lie algebras are Auslander regular, we refer for example to chapter 3 in \cite{LVO} for a modern proof that also applies to several other rings.
The Auslander-Gorenstein property was used by 
Gabber to show that universal enveloping algebras $R$ of solvable finite-dimensional Lie algebras are catenary, which means that for any two prime
ideals $P$ and $Q$ of $R$ with $P \subseteq Q$, any two saturated chains of prime ideals
between $P$ and $Q$ have the same length. We refer for example to chapter 4 in \cite{Cl} for more details.
For quantum groups that are Auslander-Gorenstein and their catenary property, we refer for example to \cite{GL} and \cite{GL2}.

\subsection{Algebraic theory of differential equations, Weyl algebras and generalisations}
Weyl algebras, and more generally, the ring of $\mathbb{C}$-linear differential operators on an irreducible smooth subvariety of an affine
space are Auslander regular, see for example chapter 3 in \cite{VO}.
The theory of Auslander regular rings is used in the study of D-modules and the algebraic theory of differential equations, we refer for example to \cite{Bj1} and \cite{Qua}.
\subsection{Noetherian Hopf algebras}
The following question seems to be one of the major open problems in the homological theory of noetherian Hopf algebras, we refer for example to \cite[Section 3]{Go} for this question and related open problems. 
\begin{question} \label{Hopfquestion}
Let $R$ be a noetherian Hopf algebra. Is $R$ Auslander-Gorenstein?
\end{question}

In Krull dimension zero, this question has a positive answer thanks to the celebrated Larson–Sweedler theorem \cite{LS}, stating that every finite-dimensional Hopf algebra is a Frobenius algebra. Question
\ref{Hopfquestion} was also proven true for some other special classes, such as polynomial identity Hopf algebras, for which we refer to \cite{WZ} and \cite{RWZ}.

\subsection{Number theory and arithmetic geometry}
Auslander-Gorenstein algebras also play a significant role in arithmetic geometry in the study of Iwasawa algebras.
We refer, for example, to the survey article \cite{ArBr} and the articles \cite{CSS,Bo,Ven} for more information.

\section{Auslander-Gorenstein algebras in the finite-dimensional case}
From this section on, we focus on the theory of Auslander-Gorenstein algebras for finite-dimensional $K$-algebras $A$, where $K$ is a field.
We assume that modules are finitely generated right modules unless otherwise stated. $D=\Hom_K(-,K)$ stands for the natural duality of $\mod A$. 
We denote the full subcategory of direct sums of direct summands of a module $M$ by $\add M$, and we write $\operatorname{ind}{C}$ to denote the set of indecomposable objects in $C$ for a subcategory $C$.
A path algebra is a finite-dimensional $K$-algebra of the form $KQ$ where $Q$ is a quiver, and a quiver algebra is a finite-dimensional $K$-algebra of the form $KQ/I$ where $I$ is an admissible ideal. We refer, for example, to the textbooks \cite{ASS}, \cite{ARS}, and \cite{SkoYam} for an introduction to the representation theory of finite-dimensional algebras. We will mostly follow the notation and conventions of the textbook \cite{SkoYam}.
We denote by $P(i)$, $I(i)$ and $S(i)$ the indecomposable projective, injective and simple $A$-modules respectively corresponding to the vertex $i$ in a quiver algebra.
We first collect several equivalent characterisations of Auslander-Gorenstein algebras, as well as some open problems and conjectures.
Recall that the flat dimension coincides with the projective dimension for finitely generated modules over finite-dimensional algebras, see for example \cite[Proposition 4.1.5]{Wei}.
The \emph{grade} of a module $M$ is defined as $\grade M:= \inf \{ i \geq 0 \mid \Ext_A^i(M,A) \neq 0 \}$.
$A$ is called \emph{Iwanaga-Gorenstein} if $\idim A_A = \idim {}_AA < \infty.$
Recall that a module $M$ with minimal injective coresolution 
$$0 \rightarrow M \rightarrow I^0 \rightarrow I^1 \rightarrow I^2 \rightarrow \cdots $$
is said to have \emph{dominant dimension} $\domdim M$ equal to $n$ if $I^n$ is the first non-projective module, and we set $\domdim M= \infty$ if no such $n$ exists. The dominant dimension of an algebra $A$ is defined as the dominant dimension of the regular module. 

\begin{theorem} \label{AusGorcharacterisation}
Let $A$ be an Iwanaga-Gorenstein algebra.
The following are equivalent:
\begin{enumerate}
    \item $A$ is Auslander-Gorenstein.
    \item For every indecomposable injective $A$-module $I$, we have that $\pdim I=\grade \soc I$. 
    \item For all $i \geq 0$ and for each submodule $X$ of a module of the form $\Ext_{A^{op}}^i(C,A)$ for $C \in \mod A^{op}$ we have $\grade X \geq i.$
    \item For each $i\geq 0$ and each simple $A$-module $S$, every composition factor of $\Ext_A^i(S,A)$ has grade at least $i$.
    \item $\Ext_A^i(\Ext_A^i(f,A))$ is a monomorphism for every monomorphism $f: X \rightarrow Y$ and $i \geq 0$.
\end{enumerate}
\end{theorem}
\begin{proof}
The equivalence of (1) and (2) was established in \cite{KMT}.
The equivalence of (1) and (3) can be found as a part of Theorem 3.7 in \cite{FGR}, and (1) and (4) are equivalent by Proposition 3.4 of \cite{AR}.
The equivalence of (1) and (5) was proven in Theorem 4.2 of \cite{HI}.
\end{proof}

The next theorem shows that being $k$-Gorenstein is left-right symmetric:

\begin{theorem}
\cite[Theorem 3.7]{FGR}
Let $k \geq 1$.
$A$ is $k$-Gorenstein if and only if $A^{op}$ is $k$-Gorenstein.
\end{theorem}

In particular, being Auslander-Gorenstein is left-right symmetric.
Recall that an algebra is said to satisfy the \emph{Auslander condition} if it is $n$-Gorenstein for all $n \geq 1$. 
An important result for algebras satisfying the Auslander condition is that they satisfy the Gorenstein Symmetry Conjecture, which states that $\idim A_A < \infty$ if and only if $\idim _{A}A < \infty$. The Gorenstein Symmetry Conjecture is still open and is a consequence of the finitistic dimension conjecture.
\begin{theorem} \label{GorsymconjAG}
\cite[Corollary 5.5]{AR}
Let $A$ satisfy the Auslander condition. Then $\idim A_A < \infty$ if and only if $\idim _{A}A < \infty$.
\end{theorem}

The following is the main open homological conjecture related to Auslander-Gorenstein algebras:
\begin{conjecture} \label{Auslander-Reiten Conjecture}
Let $A$ satisfy the Auslander condition, then $A$ is Auslander-Gorenstein.
\end{conjecture}

Auslander and Reiten formulated this as a question in \cite[Chapter 5]{AR}. In the same paper, they also raised the question whether $A$ being $k$-Gorenstein for all $k \geq 1$ implies that the finite finitistic dimension of $A$ is finite. Under the assumption that $A$ is $k$-Gorenstein for all $k \geq 1,$ being Iwanaga-Gorenstein is equivalent to having finite finitistic dimension, see \cite[Theorem 4.1]{Gu}. Thus, these two questions are equivalent to each other.
We refer to Conjecture \ref{Auslander-Reiten Conjecture} as the \textit{Auslander-Reiten Conjecture}.
We show now that this conjecture sits between two of the most important homological conjectures for finite-dimensional algebras that we recall now:

\begin{conjecture}
(Nakayama Conjecture) A finite-dimensional $K$-algebra of infinite dominant dimension is selfinjective.
\end{conjecture}

\begin{conjecture}
(Generalised Nakayama Conjecture) Let $A$ be a finite-dimensional $K$-algebra. For every simple $A$-module $S$ we have $\grade S< \infty.$
\end{conjecture}

We remark that it is also unknown whether we always have $\grade M < \infty$ for every non-zero $A$-module $M$. This is called the strong Nakayama Conjecture. The Generalised Nakayama Conjecture implies the Nakayama Conjecture, and both of these conjectures are a consequence of the famous finitistic dimension conjecture, which states that every finite-dimensional $K$-algebra has finite finitistic dimension. These are explained, for example, in the survey article \cite{Yam}.
Since we did not find a suitable reference, we prove now that the Auslander-Reiten Conjecture sits between the Nakayama and the Generalised Nakayama Conjecture:
\begin{proposition}
Let $A$ be a finite-dimensional $K$-algebra.
\begin{enumerate}
    \item If $A$ satisfies the Generalised Nakayama Conjecture, then it also satisfies the Auslander-Reiten Conjecture.
    \item If $A$ satisfies the Auslander-Reiten Conjecture, then $A$ satisfies the Nakayama Conjecture.
\end{enumerate}
\end{proposition}
\begin{proof}
Let $0 \rightarrow A \rightarrow I^0 \rightarrow I^1 \rightarrow \cdots$ be a minimal injective coresolution of $A$.
\begin{enumerate}
    \item $\grade S < \infty$ for all simple modules $S$, tells us that for every simple module $S$ there is some index $i_S$ for which $\Ext_A^{i_S}(S,A) \neq 0$ .
    But $\Ext_A^{i_S}(S,A) \neq 0$ if and only if the injective envelope $I(S)$ of $S$ appears as a direct summand of $I^{i_S}$.
    Thus, every indecomposable injective $A$-module is a direct summand of some $I^i$.
    Let now $I$ be an arbitrary indecomposable injective $A$-module that is a direct summand of $I^i$. If we assume that $A$ satisfies the Auslander condition, then
    $\pdim I \leq \pdim I^i \leq i$.
    In particular, $\idim A_A = \pdim D({}_AA) < \infty$. By Theorem \ref{GorsymconjAG}, we also have $\idim {}_AA< \infty$ and $A$ is Iwanaga-Gorenstein.
    \item Assume that $A$ has infinite dominant dimension, which is equivalent to $\pdim I^i =0$ for all $i \geq 0$. In particular, $\pdim I^i \leq i$ for all $i \geq 0$ and $A$ satisfies the Auslander condition and since we assume that $A$ also satisfies the Auslander-Reiten Conjecture, $A$ is Iwanaga-Gorenstein. 
    Thus, there is a finite minimal injective coresolution
    $0 \rightarrow A \rightarrow I^0 \rightarrow I^1 \rightarrow \cdots \rightarrow I^{n-1} \xrightarrow{f} I^n \rightarrow 0.$
    However, the condition that $I^n$ is projective-injective would imply that $f$ is split, and thus, the resolution cannot be minimal unless $A$ is itself injective.
    Thus, $A$ is selfinjective.
\end{enumerate}
\end{proof}

There are two important homological bijections associated with Auslander-Gorenstein algebras. We will recall these here. The following definitions are also listed in \cite{KMT} along with further bijections related to representation theory and combinatorics. 
The first bijection is due to Auslander and Reiten:

\begin{theorem}
\cite[Proposition 5.4]{AR}]
Let $A$ be an Auslander-Gorenstein algebra. Then there is a bijection 
$$\psi: \ \{\text{indecomposable injective A-modules }\}/_{\cong} \rightarrow \{\text{ indecomposable projective A-modules }\}/_{\cong}$$ $$I\mapsto P=\Omega^{d}(I),$$ where $d$ is the projective dimension of $I$. Moreover, $P=\Omega^d(I)$ has injective dimension equal to $d$. 
\end{theorem}
The inverse of this bijection, $\psi^{-1}$ maps an indecomposable projective $A$-module $P$ to $I=\Omega^{-d}(P)$ where $d=\idim(P).$ This bijection induces a permutation on the isomorphism classes of the simple $A$-modules, or equivalently, on their index set. 
We define the \emph{Auslander-Reiten permutation} of an Auslander-Gorenstein algebra with $n$ simple modules as $\hat{\psi}: \{1,\ldots,n\} \rightarrow \{1,\ldots,n\}$ where $\hat{\psi}(i)=j$ if $\psi(I(i))=P(j)$ for indecomposable injective modules $I(i)$ and indecomposable projective modules $P(j)$.

There is another homological bijection on the simple modules of an Auslander-Gorenstein algebra, which  was first discovered by Iyama:
\begin{theorem}
\cite[Theorem 2.10]{Iyasym}
Let $A$ be an Auslander-Gorenstein algebra. The following gives a well-defined bijection between the simple $A$-modules: $$\phi: \{\text{ simple A-modules }\}/_{\cong} \rightarrow \{ \text{ simple A-modules } \}/_{\cong}$$  $$S\mapsto \top(D\Ext_A^{g_S}(S,A)),$$ where $g_S:=\inf \{i \geq 0 \mid \Ext_A^i(S,A) \neq 0 \}$ is the grade of $S$. 
\end{theorem}
We call this bijection the \emph{grade bijection} of the Auslander-Gorenstein algebra $A$.
The cograde of the simple module $S$ is defined dually, i.e. as $\inf \{i \geq 0 \mid \Ext_A^i(D(A),S) \neq 0 \}.$ An important property of $\phi$ is that the cograde of $\phi(S)$ equals the grade of $S$.

This also induces a permutation on the index set of the simple modules. We define the \emph{grade permutation} of an Auslander-Gorenstein algebra with $n$ simple modules as $\hat{\phi}: \{1,\ldots,n\} \rightarrow \{1,\ldots,n\}$ where $\hat{\phi}(i)=j$ if $\phi(S(i))=S(j)$.

We have seen two permutations on the index set of the simple modules of an Auslander-Gorenstein algebra.
Our next theorem states that these two permutations coincide:
\begin{theorem} \label{ARpermequalgradeperm}
\cite[See Theorem 1.1]{KMT}
Let $A$ be an Auslander-Gorenstein algebra.
Then the Auslander-Reiten permutation coincides with the grade permutation.
\end{theorem}

Later we will see several explicit descriptions of the Auslander-Reiten permutation for special classes of algebras, such as incidence algebras of distributive lattices, Nakayama algebras, and gentle algebras.
Motivated by the Auslander-Reiten bijection for Auslander-Gorenstein algebras, we give the following definition:

\begin{definition}
Let $A$ be a finite-dimensional $K$-algebra.
We say that $A$ has a \emph{well-defined Auslander-Reiten map} if every indecomposable injective $A$-module $I$ has a finite minimal projective resolution and the last term of this resolution is an indecomposable projective module. 
\end{definition}
For example, every Nakayama algebra that is Iwanaga-Gorenstein has a well-defined Auslander-Reiten map.
If $A$ is an algebra with $n$ simple modules enumerated by $\{1,\ldots,n\}$ with a well-defined Auslander-Reiten map, then the \emph{Auslander-Reiten map} of $A$ is defined as the map $\zeta: \{1,\ldots,n \} \rightarrow \{1,\ldots,n\}$ where $\zeta(i)=j$ if the last non-zero term in the minimal projective resolution of the indecomposable injective module $I(i)$ is isomorphic to $P({\zeta(i)}).$
We pose the following conjecture:

\begin{conjecture}
\label{conj::AGiffARbijwelldef}
Let $A$ be a finite-dimensional $K$-algebra over an algebraically closed field $K$.
Then $A$ is Auslander-Gorenstein if and only if $A$ has a well-defined Auslander-Reiten map that is a bijection. 
\end{conjecture}
We will see in a later chapter that this conjecture has a positive answer for monomial algebras and for incidence algebras of lattices.

The Auslander-Gorenstein property is a condition that can be checked by looking at the minimal injective coresolution of the regular module. We will now see that being Auslander-Gorenstein has several global implications for the module category.

\begin{definition}
Let 
$$0 \rightarrow A \rightarrow I^0 \rightarrow I^1 \rightarrow \cdots$$ be a minimal injective coresolution of $A$. An integer
$\ell \geq 0$ is a \emph{dominant number} of $A$ if $\pdim I^i < \pdim I^{\ell}$ for any $0 \leq i < \ell$. 
\end{definition}

The dominant numbers of an Auslander-Gorenstein algebra give global information about the module category of $A$:
\begin{theorem}
Let $A$ be Auslander-Gorenstein.
Then the set of dominant numbers of $A$ coincides with the set $\{ \grade X \mid X \in \mod A \}$ as well as with the set $\{ \grade S \mid S \ \text{is simple} \}$.
\end{theorem}
\begin{proof}
This is a special case of the theorem in Section 1.1 of \cite{Iyasym}, combined with the fact that in an Auslander-Gorenstein algebra we have that $\grade Y= \min \{ \grade X, \grade Z \}$, when there is a short exact sequence $0 \rightarrow Z \rightarrow Y \rightarrow X \rightarrow 0$, see for example \cite[Chapter III.2.6]{LVO}.
\end{proof}
Thus, the dominant numbers coincide with the grades of the simple $A$-modules, and thus also with the projective dimension of the indecomposable injective $A$-modules by part (2) of Theorem \ref{AusGorcharacterisation}.
We will later see that for category $\mathcal{O}$ the dominant numbers have a deep connection to invariants appearing in geometric representation theory.

Another pleasant property of module categories of Auslander-Gorenstein algebras is the explicit knowledge of minimal approximations of the subcategories $\mathcal{X}_n:=\add (\Omega^n(\mod A) \cup \proj A)$, where $\Omega^n(\mod A)$ denotes the full subcategory of $n$-th syzygy modules.
In fact, this property even characterises algebras satisfying the Auslander condition, as was shown by Auslander and Reiten. For every $A$-module $X$ and $d \geq 0$, there exists a canonical map $\Omega^{d}(\Omega^{-d}(X)) \oplus P \rightarrow X$ where $P$ is a projective module, see \cite[Theorem 1.2]{AR} for details. This is a $\mathcal{X}_d$-approximation and we call this the \emph{canonical $\mathcal{X}_d$- approximation} of $X$.
\begin{theorem}
\cite[Section 1 and 4]{AR}
Let $A$ be an algebra. Then the following are equivalent:
\begin{enumerate}
    \item For every $A$-module $X$ and $d \geq 0$, the canonical $\mathcal{X}_d$-approximation is a minimal $\mathcal{X}_d$-approximation.
    \item $A$ satisfies the Auslander condition.
\end{enumerate}
Furthermore, in this case we have that $\mathcal{X}_d=\add (\Omega^d(\mod A))= \Omega^d(\mod A)$ is a contravariantly finite subcategory of $\mod A$.
\end{theorem}

We refer to \cite{AR} for further properties and characterisations of homological properties using subcategories associated with syzygies.
We note that the study of the subcategories $\mathcal{X}_d$ is a central topic in representation theory. For example when $A$ is Iwanaga-Gorenstein with $\idim A_A=d$, then $\mathcal{X}_d$ coincides with the subcategory of maximal Cohen-Macaulay modules (also known as Gorenstein projective modules), which is the central topic of study for Gorenstein homological algebra.
We refer for example to the survey articles \cite{ChenGor} and \cite{GLZ}.

\subsection{Generalisation of Auslander-Gorenstein algebras}
We briefly mention some generalisations of Auslander-Gorenstein algebras for finite dimensional algebras in the literature.
One generalisation comes from relaxing the condition $\pdim I^i \leq i$ in the defintion of the Auslander condition by allowing a condition of the form $\pdim I^i \leq a(i)$ for some function $a : \mathbb{N} \rightarrow \mathbb{N}$.
The most prominent generalisation is the one with $a(i)=i+1$, which leads to quasi-$n$-Gorenstein algebras, we refer to \cite{HI} and \cite{IyamaZhang} for more information.
Another recent generalisation comes from relaxing the Iwanaga-Gorenstein condition and allowing the more general Cohen-Macaulay property for Artin algebras in the sense of Auslander and Reiten \cite{ARCohen}. In \cite{CIMminimalCM} we introduce algebras satisfying the Auslander condition that are additionally Cohen-Macaulay and minimal Auslander-Cohen-Macaulay algebras and in \cite{CIMcontracted} we found many examples of such algebras inside the class of contracted preprojective algebras of Dynkin type.
\section{Blocks of category $\mathcal{O}$}

In this section we assume that the field $K= \mathbb{C}$ is the field of complex numbers and $\mathfrak{g}$ denotes a semi-simple Lie algebra over $K$ with a Cartan decomposition $\mathfrak{g}=\mathfrak{n}^{-} \oplus \mathfrak{h} \oplus \mathfrak{n}$, we refer for example to \cite[Section 0]{Hu} for details.
Recall that the universal enveloping algebra $U(\mathfrak{g})$ of $\mathfrak{g}$ is defined as the quotient of the tensor algebra $T(\mathfrak{g})$ by the two-sided ideal generated by elements of the form $X \otimes Y - Y \otimes X - [X,Y]$ for $X,Y \in \mathfrak{g}$.
The \emph{Bernstein-Gelfand-Gelfand category $\mathcal{O}$} associated 
to $\mathfrak{g}$ is defined as the full subcategory of $\mod U(\mathfrak{g})$ consisting of the modules $M$ with the following properties:
\begin{enumerate}
    \item $M$ is a weight module.
    \item The subspace $U(\mathfrak{n}) v$ of $M$ is finite-dimensional for each $v \in M.$
\end{enumerate}
Category $\mathcal{O}$ has many nice properties, such as being an abelian category. We refer to \cite{Hu} for more details and properties.
The blocks of category $\mathcal{O}$ are finite-dimensional algebras of finite global dimension.
It was recently proved in \cite{KMM} that these blocks are also Auslander regular:
\begin{theorem}
\cite[Theorem A]{KMM}
Let $A$ be a block of category $\mathcal{O}$ associated to the semi-simple Lie algebra $\mathfrak{g}$. Then $A$ is Auslander regular.
\end{theorem}
This answered a question by Marczinzik, who observed the Auslander regular property for such algebras in small examples where quiver and relations are known by the work of Stroppel \cite{Str}.
In the simplest non-trivial example, the regular block of category $\mathcal{O}$ for $\mathfrak{g}=\operatorname{sl}_2$ is simply the algebra $KQ/I$ with quiver $Q$:
\[\begin{tikzcd}
	1 & 2
	\arrow["a", shift left=3, from=1-1, to=1-2]
	\arrow["b", shift left=3, from=1-2, to=1-1]
\end{tikzcd}\]
and relations $I=\langle ab \rangle$.

We now explain two applications of the Auslander regular property for blocks of category $\mathcal{O}$.
The indecomposable simple modules $S(w)$ in the regular block of category $\mathcal{O}$ are parametrised by the elements $w \in W$ of the corresponding Weyl group $W$ of $\mathfrak{g}$. We denote the longest element of the Weyl group by $w_0$.
Lusztig's a-function is a function $a:W \rightarrow \mathbb{N}$ with several nice properties. We omit its definition as it is quite complicated and refer to \cite{Lu1} and \cite{Lu2}.
A homological interpretation of Lusztig's a-function was found in \cite{Maz}: 
\begin{theorem}
\cite[Theorem 20]{Maz}
Let $\mathfrak{g}$ be a semi-simple Lie algebra with Weyl group $W$ and let $A$ denote the regular block in category $\mathcal{O}$.
Then the projective dimension of the indecomposable injective $A$-module $I(w)$, the injective envelope of $S(w)$, is given by $2 a (w_0 w)$.
\end{theorem}

Combining this result with the Auslander regular property gives the following global information on the grades of modules in the module category of the regular block:

\begin{proposition}
\cite[Proposition 3.8]{KMT}
Let $A$ be the regular block of category $\mathcal{O}$ associated to a semi-simple lie algebra $\mathfrak{g}$.
The set of grades of the $A$-modules coincides with the set of values of two times Lusztig's a-function.
\end{proposition}

Recall that a finite-dimensional algebra $A$ has a \textit{simple preserving duality} if there is a contravariant anti-equivalence $H$ of $\mod A$ which preserves the isomorphism classes of simple $A$-modules.
All blocks of category $\mathcal{O}$ have a simple preserving duality.
The next proposition determines the Auslander-Reiten bijection for a class of algebras that contains blocks of category $\mathcal{O}$:
\begin{proposition}
\cite[Proposition 3.16]{KMT}
Let $A$ be an Auslander regular algebra with a simple preserving duality. Then the Auslander-Reiten permutation is the identity.
\end{proposition}

In particular, this tells us that the last term in a minimal projective resolution of an indecomposable injective module $I(w)$ is always given by $P(w)$ in blocks of category $\mathcal{O}$ for $w \in W.$
\section{Higher Auslander algebras, dominant Auslander regular algebras and cluster tilting}

An \textit{Auslander algebra}
is an algebra $A$ such that $\gldim A \leq 2 \leq \domdim A$. Note that this implies $\gldim A= \domdim A$ if $A$ is connected and non-semisimple.
When $A$ is an algebra of dominant dimension $\geq 1$, then we define the \textit{base algebra} of $A$ as the endomorphism algebra $\End_A(N)$, where $N$ is the direct sum of all indecomposable left projective-injective $A$-modules, called the left \textit{minimal faithful projective-injective $A$-module}. Note that a left minimal faithful projective-injective $A$-module can be written in the form $N=Af$ for an idempotent $f$ and then $N=Af$ is a right $fAf$-module as $\End_A(Af)=fAf$.
The following is the celebrated Auslander correspondence \cite{AusQu} that was a starting point for many modern developments in representation theory:
\begin{theorem}
There is a bijective correspondence up to Morita equivalence between the classes of representation-finite algebras $B$ and Auslander algebras $A$.
It is given by associating to $B$ the endomorphism algebra $A=\End_B(M)$ where $M$ is the direct sum of all indecomposable $B$-modules and associating to $A$ the base algebra $B$ of $A$.
\end{theorem}

Around 2007, this correspondence was generalised by Iyama to the generalised higher Auslander correspondence \cite{Iyamahigher}.
Recall that a $B$-module $M$ is called \textit{$n$-cluster tilting} for $n \geq 1$ if we have for the subcategories $M^{\perp n-1}:= \{X \in \mod B \mid \Ext_B^i(M,X)=0, \ i=1,..,n-1 \}$ and
$^{\perp n-1} M:= \{ X \in \mod B \mid \Ext_B^i(X,M)=0, \ i=1,..,n-1 \}$ that $M^{\perp n-1}=^{\perp n-1}M=\add M$.
Note that for $n=1$ this means that $\add M= \mod B$ and thus $B$ has to be representation-finite with $M$ containing all indecomposable $B$-modules as a direct summand. Note also that $M$ being $n$-cluster tilting implies that $M$ contains every indecomposable projective and every indecomposable injective module as a direct summand, that is $M$ is a \textit{generator-cogenerator} of $\mod B$.
An algebra $A$ is called \textit{$n$-Auslander algebra} for $n \geq 2$ if $\gldim A \leq n \leq \domdim A$. When the specific $n$ is not relevant for us, we will call an $n$-higher Auslander algebra simply \textit{higher Auslander algebra}. The higher Auslander correspondence states:
\begin{theorem}
There is a bijective correspondence up to Morita equivalence between the classes of algebras $B$ with an $n$-cluster tilting module $M$ and $(n+1)$-Auslander algebras.
The correspondence is given by associating to $(B,M)$ the endomorphism algebra $A=\End_B(M)$ and associating to $A$ the tuple $(B,N)$, where $B=\End_A(N)$ is the base algebra of $A$ and $N$ the minimal faithful projective-injective $A$-module viewed as a $B$-module.
\end{theorem}
For an elementary proof and generalisations, we refer for example to the survey \cite{Iyamarevisited}.
Clearly, higher Auslander algebras are Auslander regular.
This raises the natural question: What is the Auslander-Reiten bijection for higher Auslander algebras?
The \textit{higher Auslander-Reiten translate} $\tau_n$ is defined as $\tau_n :=\tau \Omega^{n-1}$, where $\tau$ is the classical Auslander-Reiten translate. 
There is a natural operation, that we denote by $\kappa_M$ on the indecomposable summands of a cluster tilting module $M$: An indecomposable non-projective summand $N$ of $M$ is sent to $\kappa_M(N)=\tau_n(N)$ and an indecomposable projective summand $P$ of $M$ is sent to the indecomposable injective module $\kappa_M(P)=\nu_B(P)=D \Hom_B(P,A)$. Basically, $\kappa_M$ coincides with the higher Auslander-Reiten translate on indecomposable non-projective modules, while on the indecomposable projective modules it acts as the Nakayama functor.
Note that in the endomorphism algebra $A=\End_B(M)$ of an $n$-cluster tilting module $M$, the indecomposable projective $A$-modules are given by $L_N:=\Hom_B(M,N)$ and the indecomposable injective $A$-modules are given by $T_N:=D \Hom_B(N,M)$, parametrised by the indecomposable direct summands $N$ of $M$.
Then we can describe the Auslander-Reiten bijection for higher Auslander algebras as follows, see Theorem 4.6 in \cite{MTY} in combination with the fact that the grade bijection coincides with the Auslander-Reiten bijection:
\begin{theorem}
Let $B$ be an algebra with $n$-cluster tilting module $M$ and $A=\End_A(M)$.
For an indecomposable direct summand $N$ of $M$, we have for the Auslander-Reiten bijection
$\psi_A(T_N)=L_{\kappa_M(N)}$.
\end{theorem}
In other words: The Auslander-Reiten bijection of a higher Auslander algebra corresponds to the higher Auslander-Reiten translate in the cluster-tilting subcategory.
For example, in the classical case of Auslander algebras, the Auslander-Reiten bijection recovers the classical Auslander-Reiten translate for representation-finite algebras.
Cluster tilting modules play an important role in modern representation theory with many connections to other fields such as Lie theory, algebraic geometry, cluster algebras, combinatorics... We refer for example to the survey articles \cite{Iyamarevisited}, \cite{GLSsurvey} and \cite{Asssurvey} for more information and examples.
We include here the following recent class of examples of higher Auslander algebras, that also gave a new homological characterisation of path algebras of Dynkin type. The \textit{SGC-extension} (SGC is short for smallest generator-cogenerator) of an algebra $A$ is the endomorphism algebra $S_1(A):=\End_A(D(A) \oplus A)$. We can then define inductively the $n$-th SGC-extension algebra $S_n(A)$ of $A$ as $S_n(A):=S_1(S_{n-1}(A))$.
\begin{theorem}
\cite[Theorem 1.1]{CIMserre}
Let $A=KQ$ be a path algebra of a quiver $Q$. Then the following are equivalent:
\begin{enumerate}
    \item $A$ is of finite-representation-type, that is $Q$ is of Dynkin type.
    \item There exists an $n \geq 1$ such that $S_n(A)$ is a higher Auslander algebra.
    \item There exists infinitely many $n$ such that $S_n(A)$ is a higher Auslander algebra.
\end{enumerate}
\end{theorem}
We refer to the article \cite{CIMserre} for explicit examples and generalisations. 
Thus, there is an infinite series of higher Auslander algebras $S_n(KQ)$ associated to every path algebra $KQ$ of Dynkin type. We remark that the global dimension of the algebras $S_n(KQ)$ is strictly increasing in $n$.

A recent result gave a new characterisation of higher Auslander algebras using global properties of their modules as follows:
\begin{theorem} \label{higherauscharacterisation}
Let $A$ be a connected non-semisimple algebra of finite global dimension $g$.
Then the following are equivalent:
\begin{enumerate}
    \item $A$ satisfies $\domdim A= \gldim A \geq 1$.
    \item For every module $M$ that is not projective-injective, we have $\domdim M+\pdim M=g$.
\end{enumerate}
\end{theorem}
That (2) implies (1) follows almost immediately from the definition by substituting the indecomposable projective non-injective modules for $M$. The direction (1) implies (2) can be found in Corollary 4.8 of \cite{ChenIyengarM}.
Note that the algebras with $\gldim A=\domdim A=1$ are exactly the algebras of upper triangular $m \times m$-matrices over a division ring up to Morita equivalence, see Proposition 1.17 in \cite{Iyamaclusternrep}.
The formula $\domdim M+\pdim M=g$ for higher Auslander algebras of global dimension $g$ can be seen as a non-commutative analogue for Artin algebras of the classical Auslander-Buchsbaum formula in commutative algebra, see also \cite{CMAusBuchs} for more information and generalisations.
This raises the following question:
\begin{question}
Is there a generalisation of Theorem \ref{higherauscharacterisation} for Auslander regular algebras?
\end{question}

Recently, higher Auslander algebras were generalised in \cite{CIMdom} to the class of \textit{dominant Auslander regular} algebras $A$, defined by the condition that every indecomposable projective $A$-modules $P$ satisfies $\idim P \leq n_P \leq \domdim P$ for some natural number $n_P \geq 1$
that depends on $P$.
There is also a generalisation of cluster tilting modules: A $B$-module $M$ that is a generator-cogenerator is called \textit{mixed cluster tilting} if it satisfies the following conditions:
\begin{enumerate}
\item  For each indecomposable non-projective direct summand $X$ of $M$, there exists $\ell_X \geq 1$ such that $\Ext_B^i(X,M)=0$ for all $1 \leq i < \ell_X$ and $\tau_{\ell_X}(X) \in \add M$.
\item $\bigcap\limits_{X \in \operatorname{ind}(\add M)}^{}{X^{\perp \ell_X -1} = \add M}$.
\end{enumerate}

Then we can state the following generalisation of the higher Auslander correspondence:
\begin{theorem}
\cite[Theorem 3.9]{CIMdom}
There is a bijective correspondence up to Morita equivalence between the classes of algebras $B$ with a mixed cluster tilting module $M$ and dominant Auslander regular algebras $A$ of dominant dimension at least 2.
The correspondence is given by associating to $(B,M)$ the endomorphism algebra $A=\End_B(M)$ and associating to $A$ the tuple $(B,N)$, where $B=\End_A(N)$ is the base algebra of $A$ and $N$ the minimal faithful projective-injective $A$-module viewed as a $B$-module. 
\end{theorem}

For the rest of this section, assume that the field $K$ is algebraically closed and is of characteristic $\neq 2$. This condition ensures that Auslander algebras are Koszul.
Our main motivation for the introduction of dominant Auslander regular algebras was the following question by Green, see Section 5 in \cite{Green}:
\begin{question} 
Assume $E$ is a Koszul algebra satisfying the following properties:
\begin{enumerate}
    \item The Loewy length of $E$ is 3.
    \item Each indecomposable projective $E$-module of Loewy length 3 is injective.
    \item $E$ has positive dominant dimension.
\end{enumerate}
What extra conditions imply that the Koszul dual of E is an Auslander algebra? 
\end{question}
For details on Koszul algebras and their quadratic dual, we refer for example to \cite{GreenMVsurvey}.
Using dominant Auslander regular algebras we were able to answer this question by replacing condition (3) by the stronger condition of being dominant Auslander regular:

\begin{theorem}
Assume $E$ is a Koszul algebra satisfying the following properties:
\begin{enumerate}
    \item The Loewy length of $E$ is 3.
    \item Each indecomposable projective $E$-module of Loewy length 3 is injective.
    \item $E$ is dominant Auslander regular.
\end{enumerate}
Then the Koszul dual of $E$ is an Auslander algebra.
\end{theorem}
A proof of this was part of the first version of the article \cite{CIMdom} but we decided to put it in a future forthcoming article where we will discuss Koszul duality and connections with the Auslander regular property in a greater generality \cite{CIMfuture}.

As a final remark for this section, we note that all the notions in this section have an Iwanaga-Gorenstein analogue for algebras of infinite global dimension, we refer to \cite{IyamaSolberg} and \cite{CIMdom} for more details.

\section{Auslander regular algebras of global dimension at most two and diagonal Auslander regular algebras}
It is a natural problem to classify Auslander regular algebras. This was done by Iyama for global dimension at most 2 and this can be seen as another way of generalising the classical Auslander correspondence.
We first recall some basics on the theory of torsion classes. A subcategory $\mathcal{T}$ of $\mod A$ is called a \textit{torsion class} if it is closed under factor modules and extensions and a subcategory $\mathcal{F}$ of $\mod A$ is called a \textit{torsionfree class} if it is closed under submodules and extensions. A pair $(\mathcal{T},\mathcal{F})$ of full subcategories of $\mod A$ is called a \textit{torsion theory} if $\mathcal{F}$ is torsionfree and $\mathcal{T}=\{ X \in \mod A \mid \Hom_A(X,Y)=0 \ \text{for all} \ Y \in \mathcal{F} \}$, which is then a torsion class uniquely determined by $\mathcal{F}$. 
We call such a torsion theory $(\mathcal{T},\mathcal{F})$ \textit{faithful} if $A \in \mathcal{F}$ and we call it \textit{hereditary} if $\mathcal{T}$ is closed under submodules.
For more information and examples of torsion theories, we refer for example to Chapter VI. in \cite{ASS}. 
We can now state the classification of Auslander regular algebras of global dimension at most two:
\begin{theorem}
\cite[Theorem 3.1]{Iyamaausreg2}
There is a bijective correspondence up to Morita equivalence between Auslander regular algebras $A$ of global dimension at most two and algebras $B$ with a torsion theory $(\mathcal{T},\mathcal{F})$ that is faithful and hereditary and $\mathcal{F}$ is representation-finite. 
\end{theorem}
In this correspondence, we have $A \cong \End_B(M)$ with $\mathcal{F}=\add M$. For more details of this correspondence and proofs we refer to \cite[Chapter 3]{Iyamaausreg2}. In a forthcoming work \cite{IKKM} this correspondence will be generalised to a correspondence for general 2-Gorenstein algebras, which has several applications such as a categorical version of Birkhoff's classical correspondence between posets and distributive lattices.

In \cite{Iyamaausreg2}, Iyama introduced the notion of \textit{right diagonal Auslander regular algebras}, defined as an Auslander regular algebra $A$ with minimal injective coresolution 
$$0 \rightarrow A \rightarrow I^0 \rightarrow I^1 \rightarrow \cdots \rightarrow I^n \rightarrow 0$$ 
such that every indecomposable direct summand $X$ of $I^i$ has projective dimension equal to $i$ for all $i \geq 0$.
Dually, one defines left diagonal Auslander regular algebras by looking at a minimal injective coresolution of the regular left module.
The next result gives a classification of right diagonal Auslander regular algebras of global dimension at most 2:

\begin{theorem} \label{diagausreg2}
\cite[Theorem 3.4]{Iyamaausreg2}
There is a bijective correspondence up to Morita equivalence between right diagonal Auslander regular algebras $A$ of global dimension at most 2 and algebras $B$ with only finitely many indecomposable modules with a projective socle.
\end{theorem}
In this correspondence, we have $A \cong \End_B(M)$ when $M$ is the direct sum of all indecomposable $B$-modules with a projective socle, we refer to \cite[Chapter 3]{Iyamaausreg2} for more details and proofs.
We present a simple example:
\begin{example}
Let $B=KQ$ with $Q$ the quiver
\[\begin{tikzcd}
	1 & 2 & 3
	\arrow[from=1-1, to=1-2]
	\arrow[from=1-3, to=1-2]
\end{tikzcd}\]
There are 6 indecomposable $B$-modules and four of them have a projective socle, namely those with dimension vectors $[ 1, 1, 1 ], [ 0, 1, 1 ], [ 1, 1, 0 ]$ and $[ 0, 1, 0 ]$. Let $M$ be the direct sum of these 4 indecomposable modules with a projective socle and $A = \End_B(M)$.
Then $A=KQ_2/I$ is given by quiver and relations as $Q_2=$
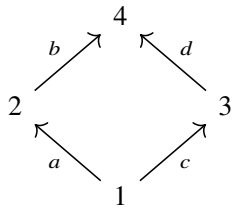
\[\begin{tikzcd}
	& 4 \\
	2 && 3 \\
	& 1
	\arrow["b", from=2-1, to=1-2]
	\arrow["d"', from=2-3, to=1-2]
	\arrow["a", from=3-2, to=2-1]
	\arrow["c"', from=3-2, to=2-3]
\end{tikzcd}\]
and $I=\langle ab-cd \rangle.$
\end{example}

Theorem \ref{diagausreg2} provides a method to produce a large class of right diagonal Auslander regular algebras. But in fact, for global dimension $>2$, no large class of right diagonal Auslander regular algebras is known except for incidence algebras of distributive lattices that we will discuss in more depth in a later section.
Motivated by this we pose the following problem:
\begin{problem}
Find new classes of right diagonal Auslander regular algebras.
\end{problem}
Finding new classes of right diagonal Auslander regular algebras is also important because of the following open problem due to Iyama, see \cite[Question 3.6.2]{Iyamaausreg2}.
\begin{problem}
If $A$ is right diagonal Auslander regular, is $A$ then also left diagonal Auslander regular?
\end{problem}

In this section, we discussed Auslander regular algebras of very small global dimension. It seems that it is not known whether Auslander regular algebras with a given number of simple modules can have arbitrarily large global dimension. We pose this as a problem:
\begin{problem}
Let $A$ be an Auslander regular algebra with $n$ simple modules. Is the global dimension of $A$ bounded by a polynomial function $f(n)$ depending only on $n$?
\end{problem}
We do not know if we can always choose $f(n)=2n$ for example.
Note that if we replace the Auslander regular condition by just having finite global dimension, then there are even algebras with only 2 simple modules and arbitrary large but finite global dimension, see for example \cite{Kiku}.

\section{Auslander-Gorenstein algebras among monomial algebras}

In this chapter, we focus on the Auslander-Gorenstein property within the class of monomial algebras. This has been studied in \cite{Kla}. A finite-dimensional quiver algebra $A=kQ/I$ is called \emph{monomial} if $I$ is generated by a set of paths, rather than by linear combinations of paths. 
This class includes well-known classes such as Nakayama algebras and gentle algebras, but is significantly more general. For example, there are no restrictions on the form a quiver $Q$ can take in a monomial algebra. In general, determining whether a quiver algebra is Auslander-Gorenstein is a  difficult task. However, the restriction on the ideal $I$ makes monomial algebras a more accessible class to study.

\begin{theorem}\cite[Corollary 4.2]{Kla}
    Every Auslander-Gorenstein monomial algebra is a string algebra. 
\end{theorem}

This result shows that being Auslander-Gorenstein is a very restrictive condition for monomial algebras. String algebras have well-understood module categories; their indecomposable modules can be described combinatorially in terms of strings and bands, see \cite{BR} and \cite{CB}. 

Gentle algebras form an important subclass of string algebras. They were introduced by Assem and Skowro\'nski in \cite{AS} to study iterated tilted algebras of type $\tilde{A}$, and have attracted significant attention since then. Remarkably, for this particular class, the Auslander-Gorenstein property can be expressed as a simple combinatorial condition on the underlying quiver.

\begin{theorem}\cite[Theorem 3.3]{Kla}
\label{thm::AGGentleClassification}
 Let $A=KQ/I$ be a gentle algebra. Then $A$ is Auslander-Gorenstein if and only if for every vertex $x$ of $Q$ $\deg^{in}(x)=2$ if and only if $\deg^{out}(x)=2.$ 
 \\
 Here $\deg^{in}(x)$ (resp. $\deg^{out}(x)$) denotes the in-degree (resp. out-degree) of a vertex $x$ in $Q.$ 
\end{theorem}

An explicit description of 2-Gorenstein monomial algebras in terms of quivers and relations is given in \cite[Theorem 4.1]{Kla}. 
According to this result, every 2-Gorenstein monomial algebra can be viewed as a collection of Nakayama algebras glued together at degree 4 vertices in a specific way. 
This insight suggests a natural approach: to "cut" the algebra at its degree 4 vertices, effectively reducing it to a disjoint union of Nakayama algebras. Surprisingly, this reduction process does not affect many of the relevant homological properties of the original algebra, leading to the following result.

\begin{theorem}\cite[Corollary 5.2]{Kla}
      We can reduce the problem of classifying all Auslander-Gorenstein monomial algebras to classifying all Auslander-Gorenstein \newline Nakayama algebras.
\end{theorem}

This reduction is significant because Nakayama algebras are far more restrictive than general monomial algebras. 
Indeed, a quiver algebra $KQ/I$ is  \textit{Nakayama} if and only if $Q$ is either a linearly oriented line or a linearly oriented cycle. 
Every indecomposable module over a Nakayama algebra is uniserial, which often enables explicit calculations and makes these algebras among the most accessible test classes for studying representation-theoretic questions. 
These algebras are closely related to Dyck paths, so they carry a rich combinatorial structure. 
Using this connection to combinatorics, Rubey and Stump characterised all 2-Gorenstein linear Nakayama algebras in \cite{RS} as those corresponding to Dyck paths containing no double deficiencies. Moreover, they proved that these algebras are enumerated by Motzkin paths, thereby proving a conjecture of Marczinzik.
In \cite{STZ}, Sen, Todorov, and Zhu investigate certain subclasses of Auslander-Gorenstein Nakayama algebras and provide a classification of cyclic Nakayama algebras that are minimal Auslander-Gorenstein, as well as dominant Auslander regular Nakayama algebras of global dimension three. 
Despite the relatively simple structure of Nakayama algebras, a full classification of Auslander-Gorenstein Nakayama algebras remains unknown to this day. 

One of the most significant results regarding the Auslander-Gorenstein property of monomial algebras is the fact that Conjecture \ref{conj::AGiffARbijwelldef} holds for every monomial algebra. This provides good evidence for the conjecture and yields a new homological characterisation of the Auslander-Gorenstein property in this setting.

\begin{theorem}\cite[Theorem 7.9]{Kla}
    Let $A=kQ/I$ be a monomial algebra. Then $A$ is Auslander-Gorenstein if and only if $A$ has a well-defined Auslander-Reiten map that is a bijection. 
\end{theorem}

Using the classification result of Auslander-Gorenstein gentle algebras, one can describe the Auslander-Reiten permutation explicitly. For this, we recall that in a gentle algebra $A=kQ/I$, a  path $w=\beta_1\ldots\beta_m$ with $\beta_i\in Q_1$ is called \textit{critical} if $\beta_i\beta_{i+1}\in I$ for every $1 \le i\le m-1.$ 

\begin{theorem}\cite[Corollary 3.11]{Kla}
\label{thm::ARpermGentle}
    Let $A=kQ/I$ be an Auslander-Gorenstein gentle algebra, and $\hat{\psi}$ its Auslander-Reiten bijection. Let $v\in Q_0.$ We can describe $\hat{\psi}$ explicitly as follows: 
    \begin{enumerate}[\rm (i)]
        \item If $\deg(v)=4$ or $0$ then $\hat{\psi}(v)=v.$

\vspace{5 pt}
{

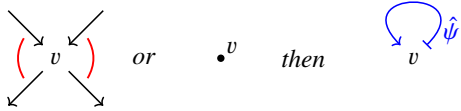
    \resizebox{0.6\linewidth}{!}{
    \begin{tikzpicture}[font=\Large, line width=0.8 pt]

    \node (BA) at (-0.5, -0.5) {};
    \node (CA) at (0.5, -0.5) {};
    \node (DA) at (-0.5, 0.5) {};
    \node (EA) at (0.5, 0.5) {};

    \draw[-,red, bend right, line width=1pt] (CA) to (EA);
    \draw[-,red, bend right, line width=1pt] (DA) to (BA);

    \node (A) at (0, 0) {$v$};
    \node (B) at (-1, -1) {};
    \node (C) at (1, -1) {};
    \node (D) at (-1, 1) {};
    \node (E) at (1, 1) {};

    \draw[->] (A) to (B){};
    \draw[->] (A) to (C){};
    \draw[->] (E) to (A){};
    \draw[->] (D) to (A){};

    \node (I) at (1.6, 0) {or};
    
    \node (F) at (3.2, 0.2) {$v$};
    \filldraw (3,0) circle (1.5pt);

    \node (I2) at (4.5, 0) {then};

    \node (I2) at (6.5, 0) {$v$};
    \node (J) at (7.2, 0.5) {\textcolor{blue}{$\hat{\psi}$}};
    
    \draw[|->, loop, blue, line width=0.8pt]  (I2) to (I2) {};

\end{tikzpicture}
    }

}

\item If $\deg^{\text{in}}(v)=1$ and there is no arrow $\beta$ such that $\alpha\beta \notin I$ where $\alpha$ is the unique arrow ending at $v$, then $\hat{\psi}(v)=s(p)$ where $p$ is the unique maximal path ending at $v$ that is not contained in $I$.

\vspace{5 pt}
{

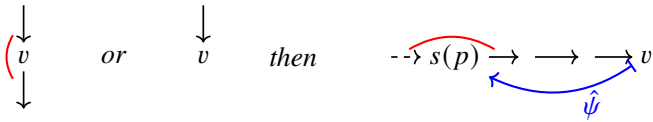
    \resizebox{0.8\linewidth}{!}{
    \begin{tikzpicture}[font=\Large, line width=0.8pt]
    
    \node (A) at (0, 0) {$v$};
    \node (B) at (0, -1) {};
    \node (D) at (0, 1) {};
    \node (BA) at (-0.1, -0.5) {};
    \node (AD) at (-0.1, 0.5) {};

    \draw[-,red, bend right, line width=1pt] (AD) to (BA);

    \draw[->] (A) to (B){};
    \draw[->] (D) to (A){};

    \node (I) at (1.5, 0) {or};

    \node (A2) at (3, 0) {$v$};
    \node (B2) at (3, 1) {};

    \draw[->] (B2) to (A2){};

    \node (I2) at (4.5, 0) {then};

    \node (K) at (6, 0) {};
    \node (L) at (7.2, 0) {{$s(p)$}};
    \node (M) at (8.4, 0) {};
    \node (N) at (9.4, 0) {};
    \node (O) at (10.4, 0) {$v$};
    \node (KL) at (6.3, 0.05) {};
    \node (LM) at (8, 0.05) {};

    \draw[->, dashed] (K) to (L){};
    \draw[->] (L) to (M){};
    \draw[->] (M) to (N){};
    \draw[->] (N) to (O){};

    \draw[-,red, bend left, line width=1pt] (KL) to (LM);

    \draw[|->,blue, bend left, line width=1pt] (O) to (L) node[midway,xshift=9.5cm,yshift=-0.8cm ] {$\hat{\psi}$};

\end{tikzpicture}
    }

}

\item If $\deg^{\text{in}}(v)=0$ and $\deg^{\text{out}}(v)=1$, or if $\deg^{\text{in}}(v)=1=\deg^{\text{out}}(v)$ and $\alpha\beta\notin I$ for the unique arrow $\alpha$  ending at $v$ and the unique arrow $\beta$ starting at $v$, then $\hat{\psi}(v)=t(w)$ where $w$ is the unique maximal critical path starting at $v$. 

{

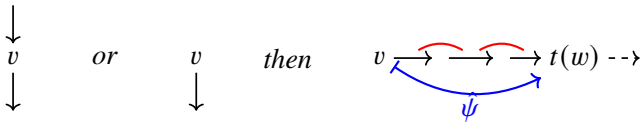
    \resizebox{0.8\linewidth}{!}{
    \begin{tikzpicture}[font=\Large, line width=0.8 pt]
    
    \node (A) at (0, 0) {$v$};
    \node (B) at (0, -1) {};
    \node (D) at (0, 1) {};
    \node (BA) at (-0.1, -0.5) {};
    \node (AD) at (-0.1, 0.5) {};

    \draw[->] (A) to (B){};
    \draw[->] (D) to (A){};

    \node (I) at (1.5, 0) {or};

    \node (A2) at (3, 0) {$v$};
    \node (B2) at (3, -1) {};

    \draw[->] (A2) to (B2){};

    \node (I2) at (4.5, 0) {then};

    \node (K) at (6, 0) {$v$};
    \node (L) at (7, 0) {};
    \node (M) at (8, 0) {};
    \node (N) at (9.2, 0) {{$t(w)$}};
    \node (O) at (10.4, 0) {};
    \node (KL) at (6.5, 0.05) {};
    \node (LM) at (7.5, 0.05) {};
    \node (MN) at (8.5, 0.05) {};
    \node (NO) at (9.7, 0.05) {};

    \draw[->] (K) to (L){};
    \draw[->] (L) to (M){};
    \draw[->] (M) to (N){};
    \draw[->, dashed] (N) to (O){};

    \draw[-,red, bend left, line width=1pt] (KL) to (LM);
    \draw[-,red, bend left, line width=1pt] (LM) to (MN);

    \draw[|->,blue, bend right, line width=1pt] (K) to (N) node[midway,xshift=7.5cm,yshift=-0.8cm ] {$\hat{\psi}$};
    
\end{tikzpicture}
    }

}

\end{enumerate}

\end{theorem}

We include Example 3.12 from \cite{Kla} here. 

\begin{example}\cite[Example 3.12]{Kla}
    Let $Q$ be the following quiver: 
    
        \vspace{10pt}
    {
    \centering

    \resizebox{0.3\linewidth}{!}{
    \begin{tikzpicture}[font=\Large, line width=0.8pt]

    \node (O) at (11, 0) {1};
    \node (P) at (11, 2) {2};
    \node (Q) at (12, 1) {3};
    \node (R) at (13, 2) {4};
    \node (S1) at (12.9, 0.1) {};
    \node (S) at (13, 0) {5};
    \node (S2) at (13.1, -0.1) {};
    \node (T1) at (11.9, -0.9) {};
    \node (T) at (12, -1) {6};
    \node (T2) at (12.1, -1.1) {};
    \node (U) at (11, -2) {7};
    \node (V) at (14, 1) {8};
    \node (X) at (15, 2) {9};

    \draw[->, line width=0.8pt, orange] (Q) to (P){};
    \draw[blue, ->, line width=0.8pt] (R) to (Q){};
    \draw[->, line width=0.8pt, blue] (Q) to (O){};
    \draw[->, line width=0.8pt, orange] (S) to (Q){};
    \draw[->, line width=0.8pt, orange] (O) to (T){};
    \draw[->, blue, line width=0.8pt] (U) to (T){};
    \draw[->, blue, line width=0.8pt] (T1) to (S1){};
    \draw[->, orange, line width=0.8pt] (T2) to (S2){};
    \draw[->, blue, line width=0.8pt] (S) to (V){};
    \draw[->, blue, line width=0.8pt] (V) to (X){};

\end{tikzpicture}

    }

}

\vspace{10 pt}
    
    In this example, the arrows are either blue or orange, and a path of length two is contained in $I$ if and only if it contains two arrows of different colours. This defines a gentle algebra $A$. Searching for a maximal path ending at a vertex (resp. a maximal critical path starting at a vertex) translates to looking for a maximal path containing arrows of only one colour ending at the given vertex (resp. a maximal path in which the arrows alternate between the two colours starting at the given vertex.)
    
According to Theorem \ref{thm::AGGentleClassification}, $A$ is Auslander-Gorenstein. Vertices 3, 5, and 6 belong to Case {\rm (i)}, 1,2, and 9 to Case {\rm(ii)}, and 4, 7, and 8 to Case {\rm(iii)} from Theorem \ref{thm::ARpermGentle}. Applying this theorem, it is easy to see that $\hat{\psi}=(142)(3)(5)(6)(789).$

\end{example}

In \cite[Proposition 2.10]{FI}, Fuller and Iwanaga discovered that if a Nakayama algebra is 2-Gorenstein, then it is necessarily 3-Gorenstein as well. This statement generalises as follows.

\begin{theorem} \cite[Corollary 7.1]{Kla}
    Let $A$ be a monomial algebra. If $A$ is $2n$-Gorenstein for some $n\ge 1$ then it is also $(2n+1)$-Gorenstein.   
\end{theorem}
We remark that for every $n \geq 0$ there exist monomial algebras (even Nakayama algebras) that are $(2n+1)$-Gorenstein but not $(2n+2)$-Gorenstein, we refer to \cite[Remark 6.3]{Kla} for details.

Finally, as an application of the reduction process described above, one obtains that monomial algebras satisfy a stronger version of Conjecture \ref{Auslander-Reiten Conjecture}. 

\begin{theorem}\cite[Corollary 7.10]{Kla}
    Let $A$ be a monomial algebra with $n$ simple modules. If $A$ is $(4n-2)$-Gorenstein, then $A$ is $(4n-2)$-Iwanaga-Gorenstein, i.e. $\idim A_A = \idim {}_AA \le 4n-2.$
\end{theorem}

\section{Auslander regular algebras among incidence algebras of posets}

We will call a partially ordered set a poset for short and assume that all posets are finite.
Recall that the \textit{Hasse quiver} $Q_P$ of a poset $P$ is by definition the quiver whose vertices $x$ correspond to elements $x \in P$ of the poset, and where $x \rightarrow y$ is an arrow in $Q_P$ if and only if $y$ covers $x$ in $P$.
For $x \in P$, $\cov(x)$ denotes the set of covers of $x$ and $\cocov(x)$ the set of cocovers of $x$.
The \textit{incidence algebra} $KP$ over a field $K$ of $P$ is defined as the quiver algebra $KQ_P/I$ where $Q_P$ is the Hasse quiver of $P$ and $I$ is generated by all expressions of the form $w_1 - w_2$ where $w_1$ and $w_2$ are paths of length $\geq 2$ starting and ending at the same vertices. By definition, in an incidence algebra, there is a unique path from $x$ to $y$, denoted by $p_y^x$, for elements $x,y \in P$ with $x \leq y$. Recall that a poset $P$ is a \textit{lattice} if every two elements $x,y \in P$ have a unique join $x \lor y$ (the smallest element $\geq x$ and $\geq y$) and a unique meet $x \land y$ (the largest element $\leq x$ and $\leq y$).
A lattice $L$ is called \textit{distributive} if we have $x \land  (y \lor z)= (x \land y) \lor (x \land z)$ for all $x,y,z \in L$.
For example the Boolean lattice of an $n$-set consisting of all subsets of $\{1,..,n\}$ is a well-known distributive lattice. We denote the unique minimal element in a lattice by $m$ and the unique maximal element by $M$.
Incidence algebras are classical topics in mathematics appearing in algebraic combinatorics \cite{Stanley} and number theory \cite{Aigner}, but recently, they have also become popular in various areas of representation theory \cite{Simson} and even some applied fields such as topological data analysis \cite{Baueretal} \cite{Bruestleetal}. 
Recently, a new connection was found between classical order theory of lattices and homological algebra: A lattice $L$ is distributive if and only if the incidence algebra $KL$ is Auslander regular. We will provide more details on this important result in the following.
In fact, for distributive lattices, one can explicitly determine the minimal injective coresolution of the regular module.
We associate to an antichain $C=\{x_1,\ldots,x_r\}$ in the lattice $L$ the submodule
\[N_C:=\sum\limits_{i=1}^{r}{p_{x_i}^m A}\]
of the projective-injective indecomposable module $P(m)$, which is the right ideal of $A$ generated by the elements $p_{x_i}^m$ for $i=1,\ldots,r$.
For an antichain $C$ in a lattice $L$, we define the \textit{antichain module} associated to $C$ by
\[M_C:=P(m)/N_C.\]
The next proposition gives an explicit (in general non-minimal) projective resolution of antichain modules:
\begin{theorem}\cite[Theorem 2.2]{IMdist} \label{antichainmoduleresolution}
Let $L$ be a lattice with incidence algebra $A$. Let $C$ be an antichain of cardinality $\ell$ with associated module $M_C$.
Then $M_C$ has a projective resolution 
\[0\to P_\ell\to\cdots\to P_0\to M_C \to0\ \mbox{ with }\ P_0=P(m) \mbox{ and } \ P_r = \bigoplus\limits_{S\subseteq C,\ |S|=r}^{}P(\bigvee S) \ \mbox{ for }\ 1\le r\le\ell.\]
\end{theorem}
We refer to \cite[Section 2]{IMdist} for a proof and more details such as the explicit form of the maps in the projective resolution.
Using the structure theory of distributive lattices, one can then obtain the following theorem that gives explicit minimal resolutions for indecomposable projective and indecomposable injective modules in the incidence algebra of a distributive lattice:

\begin{theorem}\cite[Theorem 3.2]{IMdist} \label{homologicaldimensionsandausreg}
Let $L$ be a distributive lattice with incidence algebra $A$ and $x,y\in L$. Let $P(y)$ be the unique indecomposable projective right $A$-module corresponding to the point $y$ and $I(x)$ the unique indecomposable injective right $A$-module corresponding to the point $x$. Then
\begin{enumerate}
\item[\rm(1)] $A$ is diagonal Auslander regular and has a minimal injective coresolution
$$0 \rightarrow A \rightarrow I^0 \rightarrow I^1 \rightarrow \cdots \rightarrow I^i \rightarrow \cdots $$
such that $I^i$ is a direct sum of copies of $I(x)$ for elements $x\in L$ with $|\cov(x)|=i$.
\item[\rm(2)] 
$\pdim I(x)=|\cov(x)|=:\ell$ holds, and $I(x)$ has a minimal projective resolution 
\[0\to P_\ell\to\cdots\to P_0\to I(x) \to0\ \mbox{ with }\ P_0=P(m)  \]
\[ \mbox{ and } \ P_r = \bigoplus\limits_{S\subseteq\min([m,x]^c),\ |S|=r}^{}P(\bigvee S) \ \mbox{ for }\ 1\le r\le\ell.\]
\item[\rm(3)] $\idim P(y)=|\cocov(y)|=:\ell$ holds, and $P(y)$ has a minimal injective coresolution
\[0\to P(y)\to I^0\to\cdots\to I^\ell\to0\ \mbox{ with }\ I^0=I(M) \] 
\[ \mbox{ and } \ I^r = \bigoplus\limits_{S\subseteq\max([y,M]^c),\ |S|=r}^{}I(\bigwedge S) \ \mbox{ for }\ 1\le r\le\ell.\]
\end{enumerate}
\end{theorem}
Thus, incidence algebras of distributive lattices give a new class of examples of diagonal Auslander regular algebras.
The other direction, that $KL$ being Auslander regular implies that $L$ must be distributive, follows from the following slightly stronger statement:
\begin{theorem}\cite[Theorem 5.4]{IMdist}
Let $L$ be a lattice such that $KL$ is 2-Gorenstein. Then $L$ is distributive.
\end{theorem}

The \textit{order dimension} of a poset $P$ is defined as the minimal $t$ such that $P$ can be embedded into a product of $t$ totally ordered chains. The order dimension is a fundamental notion in order theory with many mysteries and open problems surrounding it, see for example the textbook \cite{Trotter} dedicated to the order dimension. 
Suprisingly, the fundamental notion of dimension in order theory, the order dimension, and the fundamental notion of dimension in homological algebra, the global dimension, coincide in a nice way for distributive lattices:
\begin{theorem}\cite[Theorem 3.9]{IMdist}
Let $L$ be a distributive lattice with at least two elements.
Then the order dimension of $L$ coincides with the global dimension of $KL$.
\end{theorem}
In particular, this implies that the global dimension of the incidence algebra of a distributive lattice is independent of the field $K$, a result that is not true for general posets. 
The proof of the previous theorem follows from an explicit description of the minimal projective resolutions of the simple modules and a classical result of Dilworth \cite{Dil}. 
Since the incidence algebra of a distributive lattice $KL$ is Auslander regular, a natural question arises: What is the Auslander-Reiten permutation of $KL$?
There is a fundamental operation acting on the points of a distributive lattice called rowmotion bijection. First, let us recall Birkhoff's classical representation theorem for distributive lattices. This states that every distributive lattice $L$ is isomorphic to the set of order ideals of a unique poset $P$. Here, an \textit{order ideal} $I$ in a poset $P$ is a subset $I$ of $P$ such that if $y \in I$ then $x \in I$ as well for all $x \leq y$.
We can then define the \textit{rowmotion bijection,} $row$  on a distributive lattice which is given as the set of order ideals of a poset $P$ as the image of an order ideal $I$, $\row(I)$ being the order ideal generated by the minimal elements in $P \setminus I$. This is a classical bijection on the set of order ideals of $P$ and thus on the elements of the distributive lattice $L$, which was discovered several times independently since the 1970s. For a recent survey article about the rowmotion operation and dynamical algebraic combinatorics, we refer, for example to \cite{Striker}.
We can now state the following theorem:
\begin{theorem}
Let $L$ be a distributive lattice given as the set of order ideals of a poset $P$.
Then the Auslander-Reiten permutation of $KL$ on the points of $L$ coincides with the rowmotion bijection of $L$. 
\end{theorem}
This is a consequence of \cite[Theorem 4.4]{IMdist} and the fact that the Auslander-Reiten permutation coincides with the grade permutation by Theorem \ref{ARpermequalgradeperm}.

Since the classification of Auslander regular algebras among incidence algebras of lattices had such a beautiful answer, the following problem is quite natural:
\begin{problem}
Find a classification of Auslander regular posets.
\end{problem}

The problem is wide open so far, and a classification seems to be much more complicated than in the lattice case.
We list now some partial progress towards this problem. 

Recall that a poset $P$ is \textit{bounded} if it has a unique minimal element and a unique maximal element.
\begin{proposition}\cite[Proposition 5.1]{IMdist}
Let $P$ be a poset. Then $P$ is bounded if and only if $KP$ is 1-Gorenstein.
\end{proposition}

Thus, when we aim for a classification of Auslander regular posets, we can always assume that we are dealing with bounded posets. Recall that the \textit{MacNeille completion} of a poset $P$ is the smallest lattice $L$ such that $P$ embeds into $L$. $P$ is called \textit{dissective} if the MacNeille completion is distributive. We refer, for example, to \cite{Reading} for more on dissective posets.
The following classification of 2-Gorenstein incidence algebras is a work in progress \cite{IKKM}:
\begin{theorem}
Let $P$ be a bounded poset. Then $KP$ is 2-Gorenstein if and only if $P$ is dissective.
\end{theorem}
In \cite{IKKM}, we will also obtain classification results for Auslander regular incidence algebras in some special cases that will allow us to construct large new classes of Auslander regular algebras. Furthermore, we will also present a new approach using relative homological algebra that gives a categorical generalisation of Birkhoff's classical representation theorem.
Finally, let us mention the following result from \cite{IKKM}, which gives a positive answer to Conjecture \ref{conj::AGiffARbijwelldef} for lattices:
\begin{proposition}
Let $A=KL$ be the incidence algebra of a lattice $L$. Then $A$ is Auslander regular (equivalently, $L$ is distributive) if and only if $A$ has a well-defined Auslander-Reiten map that is a bijection.
\end{proposition}

\section{The Coxeter permutation as a generalisation of the Auslander-Reiten permutation}
Every Auslander regular algebra has a canonical permutation on its simple modules, namely, the Auslander-Reiten permutation. Our goal was to find a generalisation that assigns such a permutation to every finite-dimensional algebra of finite global dimension. For simplicity, we assume in this section, that our algebras are quiver algebras. Note that this is no real loss of generality when we work over an algebraically closed field $K$, since in that case every $K$-algebra is Morita equivalent to a quiver algebra.
So let $A=KQ/I$ be a connected quiver algebra with $n$ simple modules and enumerate its vertices by $\{1,\ldots,n\}$. Recall that the \textit{Cartan matrix} $\phi_A$ of $A$ is defined as the $n \times n$-matrix with entries $\phi_{i,j}=\dim e_j A e_i$, that is, the number of non-zero paths from $j$ to $i$ in $Q.$ It is well known that when $A$ has finite global dimension, then the determinant of $\phi_A$ is $\pm 1$ and thus $\phi_A $ is invertible over the integers, see \cite{Eil}.
Thus, for algebras $A$ of finite global dimension, we can define the \textit{Coxeter matrix} $C_A$ of $A$ by $C_A=- \phi_A^T \phi_A^{-1}$. The Coxeter matrix is a classical tool in representation theory and can be seen as the linear algebraic version of the Auslander-Reiten translate for hereditary algebras; we refer, for example, to \cite{ASS} for details. 
Recall that the \textit{Bruhat decomposition} of an invertible $n \times n$-matrix $M$ with entries in a field is a factorisation $M= U_1 P U_2$ with $U_1$ and $U_2$ upper triangular matrices and $P$ a permutation matrix. This factorisation always exists, and the permutation matrix $P$ is uniquely determined by $M$, see for example \cite[Chapter 4]{AlpBell}.
Given an $n \times n$ permutation matrix $P$, we define the \textit{row-permutation $p_r$} associated to $P$ as the permutation acting on $\{1,\ldots,n\}$ given by $p_r(i)=j$ if in the $i$-th row of $P$ the first non-zero entry appears in column $j$. Dually, we define the \textit{column-permutation $p_c$} associated to $P$ as the permutation acting on $\{1,\ldots,n\}$ given by $p_c(i)=j$ if in the $i$-th column of $P$ the first non-zero entry is in row $j$.
For an algebra $A$ of finite global dimension, we define the \textit{Coxeter permutation} $p_A$ as the column permutation associated to the permutation matrix $P$ in a Bruhat factorisation of the Coxeter matrix $C$. In general, the Coxeter permutation depends on the ordering of the simple modules of the algebra $A$.
Note that when $A$ is an acyclic quiver algebra, then we can order the simple modules so that the Cartan matrix of $A$ is lower triangular. We will usually assume this for acyclic quiver algebras.
The following lemma is elementary, see \cite{KKM} for a proof:
\begin{lemma}
Let $A$ be an algebra of finite global dimension with a lower triangular Cartan matrix $\phi_A$ and a Bruhat decomposition $\phi_A=\hat{U}_1 Q \hat{U}_2$.
Then the Coxeter permutation $p_A$ coincides with the row-permutation associated with $Q$.
\end{lemma}

While the Cartan matrix is usually easier to deal with than the Coxeter matrix, experiments suggest that the Bruhat factorisation of the Coxeter matrix carries a deeper meaning compared to the Bruhat factorisation of the Cartan matrix, as we will see later in some special situations.
We give an example:
\begin{example}
Let $L$ be the distributive lattice with Hasse quiver
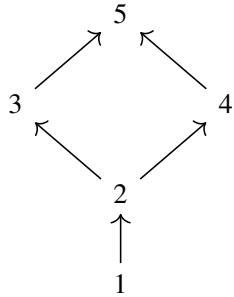
\[\begin{tikzcd}
	& 5 \\
	3 && 4 \\
	& 2 \\
	& 1
	\arrow[from=2-1, to=1-2]
	\arrow[from=2-3, to=1-2]
	\arrow[from=3-2, to=2-1]
	\arrow[from=3-2, to=2-3]
	\arrow[from=4-2, to=3-2]
\end{tikzcd}\]
and let $A=KL$ be its incidence algebra.
The Cartan matrix $\phi_A$ is given by:
$\phi_A=$
\[
\left[
\begin{array}{ccccc}
1 & 0 & 0 & 0 & 0 \\
1 & 1 & 0 & 0 & 0 \\
1 & 1 & 1 & 0 & 0 \\
1 & 1 & 0 & 1 & 0 \\
1 & 1 & 1 & 1 & 1
\end{array}
\right]
\]
A Bruhat factorisation is given by $\phi_A=$
\[
\left[
\begin{array}{rrrrr}
-1 & 0 & 0 & 0 & 1 \\
0 & 1 & -1 & -1 & 1 \\
0 & 0 & -1 & 0 & 1 \\
0 & 0 & 0 & -1 & 1 \\
0 & 0 & 0 & 0 & 1
\end{array}
\right]
\left[
\begin{array}{rrrrr}
0 & 1 & 0 & 0 & 0 \\
0 & 0 & 0 & 0 & 1 \\
0 & 0 & 0 & 1 & 0 \\
0 & 0 & 1 & 0 & 0 \\
1 & 0 & 0 & 0 & 0
\end{array}
\right]
\left[
\begin{array}{rrrrr}
1 & 1 & 1 & 1 & 1 \\
0 & 1 & 1 & 1 & 1 \\
0 & 0 & 1 & 0 & 1 \\
0 & 0 & 0 & 1 & 1 \\
0 & 0 & 0 & 0 & 1
\end{array}
\right]
\]
From this, we can read off that the Coxeter permutation of $A$ as
\(
\left(
\begin{array}{ccccc}
1 & 2 & 3 & 4 & 5 \\
2 & 5 & 4 & 3 & 1
\end{array}
\right)
\).
\end{example}
In the following, we present several recent results which indicate that the Coxeter permutation is the right generalisation of the Auslander-Reiten permutation.
Let $A$ be an algebra with $n$ simple modules, and label the vertices of its underlying quiver by $\{1,\ldots, n\}$. So the simple $A$-modules are denoted by $S(1),\ldots, S(n)$. We call the ordering of the vertices \textit{admissible} if $\grade S(i) < \grade S(j)$ implies $i>j$. Note that such an ordering always exists: Just order the simple modules in a non-increasing order according to their grades.
\begin{theorem} \cite[Theorem 3.13]{KMT}
Let $A$ be an Auslander regular algebra with an admissible ordering of the simple modules. Then the Coxeter permutation coincides with the Auslander-Reiten permutation.
\end{theorem}
See \cite[Section 3]{KMT} for a proof and more details. 
However, an admissible ordering might not always be the most natural ordering. For the incidence algebra $KP$ of a poset $P$ with $n$ elements, there exists a class of canonical orderings labeled by $\{1,\ldots,n\}$ called \textit{linear extensions}. These are defined by the condition that if $i \leq_P j$ in $P$ then $i \leq j$, where $\leq_P$ denotes the order in $P$.
In \cite{KMT}, we obtained the following new characterisation of distributive lattices using the Bruhat decomposition of the Coxeter matrix:
\begin{theorem} \cite[Theorem 1.4]{KMT}\label{thm::distributiveCharacterisationViaCoxeterMatrix}
Let $L$ be a lattice with elements ordered according to a linear extension. Let $A=KL$ be the incidence algebra of $L$ with Coxeter matrix $C_A$.
Then $C_A$ has a Bruhat decomposition of the form $C_A=U_1PU_2$ with $U_1$ being the identity matrix if and only if $L$ is Auslander regular. In this case, the Coxeter permutation coincides with the rowmotion bijection on $L$.
\end{theorem}

Note that $U_1$ is the identity matrix in some Bruhat decomposition of $C_A$ if and only if $C_A$ is an upper triangular matrix up to reordering its rows, which is a condition easily checked simply by looking at the matrix $C_A.$ 

In \cite[Question 5.3]{KMT}, we posed the question whether the previous theorem holds more generally for every bounded poset $P$. More precisely, we asked whether for such a poset $P$, $KP$ is Auslander regular if and only if $U_1=id$ in some Bruhat decomposition of $KP$, if we order the vertices of $kP$ according to a linear extension. 

Theorem \ref{thm::distributiveCharacterisationViaCoxeterMatrix} motivated us in \cite{echopaper} to study the Coxeter permutation for incidence algebras of posets in the case when simple modules are ordered using a linear extension. We remark that in \cite{echopaper} we call the (inverse of the) Coxeter permutation in the case of posets \textit{echelonmotion}, but for simplicity, we will stick to the name Coxeter permutation in this article. We call a poset $P$ \textit{echelonmotion-independent} or, for short, \textit{independent} if the Coxeter permutation is independent of the chosen linear extension.  
It is a central problem in dynamical algebraic combinatorics to generalise the rowmotion bijection from distributive lattices to more general classes of posets (see, for example, \cite{Barn}, \cite{ThomasWilliamsindependent}, \cite{ThomasWilliamsslow}). 
This motivates the following problem:
\begin{problem}
Classify the independent posets $P$.
\end{problem}
The paper \cite{echopaper} provides a complete classification for lattices, in which case the answer turns out to be quite beautiful. 
Recall that a lattice $L$ is \textit{meet-semidistributive} if for all $x,y \in L$ such that $x \leq y$, the set $\{z \in L \mid z \land y =x \}$ has a maximum element, and call $L$ \textit{join-semidistributive} if its dual is meet-semidistributive. A lattice $L$ is called \textit{semidistributive} if it is both meet-semidistributive and join-semidistributive.
Naturally, every semidistributive lattice is distributive, but the converse is not true. Here is the smallest non-distributive lattice which happens to be semidistributive:
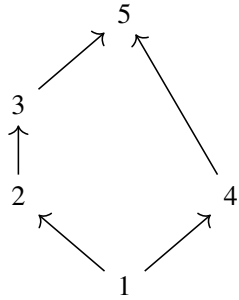
\[\begin{tikzcd}
	& 5 \\
	3 \\
	2 && 4 \\
	& 1
	\arrow[from=2-1, to=1-2]
	\arrow[from=3-1, to=2-1]
	\arrow[from=3-3, to=1-2]
	\arrow[from=4-2, to=3-1]
	\arrow[from=4-2, to=3-3]
\end{tikzcd}\]
For semidistributive lattices, a natural generalisation of the rowmotion bijection was described in \cite{Barn}. As the definition is slightly complicated, we omit it here and refer to \cite{Barn} and \cite{echopaper} for details. We also refer to the generalised rowmotion bijection for semidistributive lattices simply as rowmotion bijection in the following.
\begin{theorem} \cite[Theorem 1.3]{echopaper}
A lattice $L$ is independent if and only if $L$ is semidistributive. In this case, the Coxeter permutation with respect to a linear extension coincides with the rowmotion bijection.
\end{theorem}

Not much is known about general independent posets yet. We cite the following result from \cite{echopaper}, which gives some restrictions:
\begin{theorem} \cite[Theorems 1.6, 1.7 and 1.8]{echopaper}
Let $P$ be an independent poset. Then the following statements hold:
\begin{enumerate}
    \item $P$ is bounded.
    \item The MacNeille completion of $P$ is semidistributive.
    \item The Coxeter permutation of $P$ has no fixed points.
\end{enumerate}
\end{theorem}

The following figure illustrates an example of an independent poset that is not a lattice (this is the strong Bruhat order of the symmetric group with three elements):
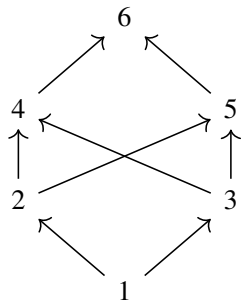
\[\begin{tikzcd}
	& 6 \\
	4 && 5 \\
	2 && 3 \\
	& 1
	\arrow[from=2-1, to=1-2]
	\arrow[from=2-3, to=1-2]
	\arrow[from=3-1, to=2-1]
	\arrow[from=3-1, to=2-3]
	\arrow[from=3-3, to=2-1]
	\arrow[from=3-3, to=2-3]
	\arrow[from=4-2, to=3-1]
	\arrow[from=4-2, to=3-3]
\end{tikzcd}\]
There are bounded posets $P$ that are not independent but whose MacNeille completion is not only semidistributive, but also distributive. We refer, for example, to \cite[Figure 3]{echopaper}.

In \cite[Conjecture 8.2]{echopaper}, we posed the following conjecture that would show that posets with Auslander regular incidence algebras are independent:
\begin{conjecture}
Let $K$ be a field of characteristic 0.
Let $P$ be a poset such that the incidence algebra $KP$ is Auslander regular. 
Then $P$ is independent and the Coxeter permutation coincides with the Auslander-Reiten permutation.
\end{conjecture}
We note that in the previous conjecture, whether the Coxeter permutation coincides with the Auslander-Reiten permutation or its inverse depends on convetions and in \cite{echopaper} we used opposite conventions to those in this survey.
This conjecture is verified for posets with at most 10 elements. We refer to \cite{echopaper} for more interesting problems and conjectures related to the Coxeter permutation.
Now we move from incidence algebras to Nakayama algebras.
In \cite{Ringel}, Ringel found a mysterious bijection on the simple modules of Nakayama algebras while studying the finitistic dimension of these algebras. For simplicity, we stick to connected quiver algebras.
Recall that a \textit{Nakayama algebra} is a quiver algebra $KQ/I$ with admissible relations $I$ and a quiver of the form (linear case)
\[\begin{tikzcd}
	1 & 2 & 3 & {n-1} & n
	\arrow[from=1-1, to=1-2]
	\arrow[from=1-2, to=1-3]
	\arrow[dashed, from=1-3, to=1-4]
	\arrow[from=1-4, to=1-5]
\end{tikzcd}\]
or of the form (cyclic case):
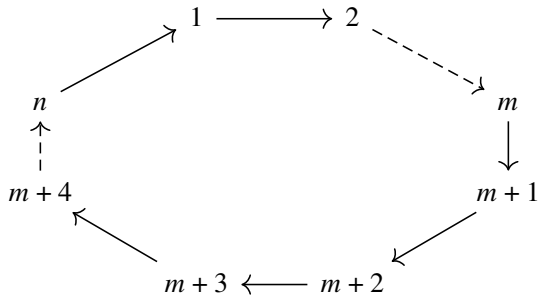
\[\begin{tikzcd}
	& 1 & 2 \\
	n &&& m \\
	{m+4} &&& {m+1} \\
	& {m+3} & {m+2}
	\arrow[from=1-2, to=1-3]
	\arrow[dashed, from=1-3, to=2-4]
	\arrow[from=2-1, to=1-2]
	\arrow[from=2-4, to=3-4]
	\arrow[dashed, from=3-1, to=2-1]
	\arrow[from=3-4, to=4-3]
	\arrow[from=4-2, to=3-1]
	\arrow[from=4-3, to=4-2]
\end{tikzcd}\]
We refer, for example, to \cite[chapter V]{ASS} for more details on Nakayama algebras.
The \textit{Kupisch series} of a Nakayama algebra is the sequence $[c_1,\ldots,c_n]$, with $c_i=\dim e_i A$ the vector space dimension of the $i$-th indecomposable projective $A$-module. A Nakayama algebra is uniquely determined up to isomorphism by its Kupisch series. 
In \cite{Ringel}, Ringel defined a bijection $h$ on the simple modules $S(1),\ldots,S(n)$ corresponding to the vertices $1,\ldots,n$ of a Nakayama algebra as follows:
Denote the injective envelope of a module $M$ by $I(M)$, and for a simple module $S$ set $e(S):= \min \{ \pdim S , \pdim I(S) \}$ and $h(S):=\top \Omega^{e(S)}(NS)$, where $NS=I(S)$ if $e(S)$ is even and $NS=S$ if $e(S)$ is odd. We call $h$ \textit{the homological bijection} of the Nakayama algebra $A$. 
We define the \textit{homological permutation} of $A$ as the map $\hat{h}: \{1,\ldots,n\} \rightarrow \{1,\ldots,n\}$ with $\hat{h}(i)=j$ if and only if $h(S(i))=S(j)$.
We refer to \cite[Section 4]{Ringel} for more details and a proof that $h$ is indeed a bijection.

Recall that a \textit{Dyck path} of semilength $n$ is a lattice path in the grid $\mathbb{Z}^2$ starting at $(0,0)$ and ending at $(2n-2,0)$ with allowed steps being diagonal up steps $U=(1,1)$ and diagonal down steps $D=(1,-1)$ such that this lattice path never goes below the $x$-axis. Dyck paths of semilength $n+1$ are enumerated by the Catalan numbers $C_n=\frac{1}{n+1} \binom{2n}{n}$.
Nakayama algebras with a linear quiver (shortly called linear Nakayama algebras in the following) with $n$ simple modules are in natural bijection to Dyck paths of semilength $n$ via the top boundary of their Auslander-Reiten quiver, we refer for example to \cite{MRS} and \cite{KMMRS} for details and applications of this correspondence. 
Here we only present one example:
\begin{example}
\label{ex::Nakayama}
Let $A=KQ/I$ be the Nakayama algebra with quiver $Q$ 
\[\begin{tikzcd}
	1 & 2 & 3 & 4
	\arrow["a", from=1-1, to=1-2]
	\arrow["b", from=1-2, to=1-3]
	\arrow["c", from=1-3, to=1-4]
\end{tikzcd}\]
and relations $I=\langle abc \rangle.$
The Kupisch series of $A$ is given by $[3,3,2,1]$ and the Auslander-Reiten quiver looks as follows with the top boundary Dyck path marked in red:
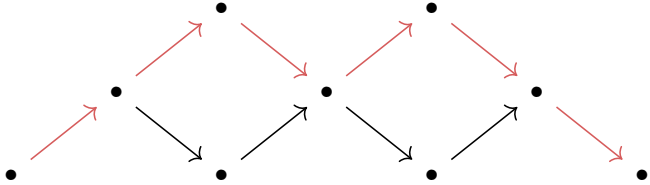
\[\begin{tikzcd}
	&& \bullet && \bullet \\
	& \bullet && \bullet && \bullet \\
	\bullet && \bullet && \bullet && \bullet
	\arrow[color={rgb,255:red,214;green,92;blue,92}, from=1-3, to=2-4]
	\arrow[color={rgb,255:red,214;green,92;blue,92}, from=1-5, to=2-6]
	\arrow[color={rgb,255:red,214;green,92;blue,92}, from=2-2, to=1-3]
	\arrow[from=2-2, to=3-3]
	\arrow[color={rgb,255:red,214;green,92;blue,92}, from=2-4, to=1-5]
	\arrow[from=2-4, to=3-5]
	\arrow[color={rgb,255:red,214;green,92;blue,92}, from=2-6, to=3-7]
	\arrow[color={rgb,255:red,214;green,92;blue,92}, from=3-1, to=2-2]
	\arrow[from=3-3, to=2-4]
	\arrow[from=3-5, to=2-6]
\end{tikzcd}\]

\end{example}

Thus, identifying linear Nakayama algebras with their corresponding Dyck paths allows us to interpret Ringel's homological permutation as a map from the set of Dyck paths to permutations. Ringel has shown in \cite[Section 4.5]{Ringel} that this map is injective.

After discovering the homological bijection for Dyck paths, Ringel posed the natural question to several experts whether there is an elementary, more combinatorial interpretation of this bijection, at least in the linear case.
We remark that there is a unique canonical choice to order the vertices of a linear Nakayama algebra, namely, label the vertices from $1$ to $n$ so that there is arrow from $i$ to $j$ if and only if $j=i+1$, as depicted in the linear quiver above.
We will always use this ordering, which allows us to speak of a unique Coxeter permutation associated to linear Nakayama algebras.
We give the following answer to Ringel's question in forthcoming work \cite{KKM}:
\begin{theorem}\label{thm::CoxeterNakayama}
Let $A$ be a linear Nakayama algebra.
Then the Coxeter permutation of $A$ coincides with Ringel's homological permutation.
\end{theorem}
We illustrate this in the following example:
\begin{example}
Let $A$ again be the linear Nakayama algebra with Kupisch series [3,3,2,1] from Example \ref{ex::Nakayama}.
\begin{enumerate}[\rm(1)]
    \item The first simple module $S(1)$ coincides with its injective envelope $I(1)$, $e(S(1))=\pdim S(1)=2$ is even, and $h(S(1))=\top \Omega^2(S(1))=S(4)$.
    \item The second simple module $S(2)$ has projective dimension 1 and its injective envelope $I(2)$ has projective dimension 2. Thus, $e(S(2))=1$ is odd and $h(S(2))=\top(\Omega^1(S(2)))=S(3).$
    \item The injective envelope of the third simple module $S(3)$ is projective and thus $e(S(3))=0$ is even and $h(S(3))=\top(I(3))=S(1)$.
    \item The simple module $S(4)$ is projective and thus $e(S(4))=0$ is even and $h(S(4))=\top(I(4))=S(2)$.
\end{enumerate}

Therefore, Ringel's homological permutation for $A$ is given by 
\[
\begin{pmatrix}
1 & 2 & 3 & 4 \\
4 & 3 & 1 & 2
\end{pmatrix}
\]
The Cartan matrix $\phi_A$ of $A$ has the form 
\[
\left[
\begin{array}{cccc}
1 & 0 & 0 & 0 \\
1 & 1 & 0 & 0 \\
1 & 1 & 1 & 0 \\
0 & 1 & 1 & 1
\end{array}
\right]
\]
and a Bruhat factorisation of $\phi_A$ is given by 
\[
\left[
\begin{array}{cccc}
1 & 0 & 1 & -1 \\
0 & -1 & 1 & 0 \\
0 & 0 & 1 & 0 \\
0 & 0 & 0 & 1
\end{array}
\right]
\cdot
\left[
\begin{array}{cccc}
0 & 0 & 0 & 1 \\
0 & 0 & 1 & 0 \\
1 & 0 & 0 & 0 \\
0 & 1 & 0 & 0
\end{array}
\right]
\cdot
\left[
\begin{array}{cccc}
1 & 1 & 1 & 0 \\
0 & 1 & 1 & 1 \\
0 & 0 & 1 & 0 \\
0 & 0 & 0 & 1
\end{array}
\right]
\]
From this, we can conclude that the Coxeter permutation coincides with Ringel's homological permutation in this example, as expected from Theorem \ref{thm::CoxeterNakayama}.

\end{example}

Similar to the case of incidence algebras of distributive lattices, we can use the Bruhat decomposition of the Coxeter matrix of a linear Nakayama algebra to decide whether this algebra is Auslander regular:
\begin{theorem}\label{thm::RingelsBij=CoxeterPerm}
Let $A$ be a linear Nakayama algebra.
Then $A$ is Auslander regular if and only if the Coxeter matrix $C_A$ has a Bruhat decomposition $C_A=U_1 P U_2$ with $U_1$ being the identity matrix.
\end{theorem}
We refer to \cite{KKM} for details.

Despite the previous results and the rich combinatorial structure of Dyck paths, an explicit enumeration of linear Nakayama algebras that are Auslander regular is not known. We pose this as a final question; see also \cite{KKM} for some partial results:
\begin{problem}
Give an explicit enumeration of Auslander regular linear Nakayama algebras with $n$ simple modules.
\end{problem}
We saw in the section about monomial algebras that the classification of monomial Auslander regular algebras is reduced to Nakayama algebras, which gives another motivation for the previous problem and also a motivation to consider the same problem for cyclic Nakayama algebras.


\begin{ack}
We thank Markus Kleinau for useful comments and suggestions.
\end{ack}

\begin{funding}
The author V. Kl\'asz was supported by the Deutsche Forschungsgemeinschaft (DFG, German Research Foundation) under Germany's Excellence Strategy grant EXC-2047/1-390685813.
\end{funding}



\end{document}